\documentclass[11pt]{amsart}
\usepackage{amsmath,amssymb,amsfonts,amsthm, amscd,indentfirst}
\usepackage{amsmath,latexsym,amssymb,amsmath,
	amscd,amsthm,amsxtra}
\usepackage{hyperref}\usepackage{url}
\usepackage{color}
\usepackage{pdfpages}
\usepackage{graphics}
\usepackage{enumitem}

\usepackage{graphicx}
\usepackage{caption}

\usepackage{epsfig,here}
\usepackage{subfigure,here}

\newtheorem{definition}{Definition}[section]
\newtheorem{theorem}{Theorem}[section]
\newtheorem{prop}{Proposition}[section]

\newtheorem{corollary}{Corollary}[section]
\def\rr{\mathbb{R}}
\def\ss{\mathbb{S}}
\def\hh{\mathbb{H}}
\def\bb{\mathbb{B}}

\def\tr{\mathrm{tr}}

\def\p{\partial}

\def\a{\alpha}

\def\th{\theta}

\def\p{\partial}

\def\S{{\Sigma}}
\def\<{\langle}
\def\>{\rangle}
\def\div{{\rm div}}
\def\n{\nabla}

\def\De{\Delta}
\def\vp{\varphi}
\def\R{{\mathbb R}}
\def\ep{\epsilon}

\numberwithin{equation} {section}

\begin{document}
	
	\title[Stable capillary hypersurfaces supported on a horosphere]{Stable capillary hypersurfaces supported on a horosphere in the hyperbolic space}
	\author{Jinyu Guo}
	\address{School of Mathematical Sciences\\
		Xiamen University\\
		361005, Xiamen, P.R. China}
	\email{jinyu.guo@math.uni-freiburg.de
	}
	\author{Guofang Wang}
	\address{Universit\"at Freiburg,
		Mathematisches Institut,
		Eckerstr. 1,
		79104 Freiburg, Germany}
	\email{guofang.wang@math.uni-freiburg.de}
	\author{Chao Xia}
	\address{School of Mathematical Sciences\\
		Xiamen University\\
		361005, Xiamen, P.R. China}
	\email{chaoxia@xmu.edu.cn}
	
	\thanks{CX is  supported by the  NSFC (Grant No. 11871406). The paper was carried out while JG was visiting the Mathematical Institute, the University of Freiburg under the support by the China Scholarship Council.}
	
	\begin{abstract}
		In this paper, we study a stability problem of free boundary hypersurfaces, and also
		capillary ones whose boundary supported on a  horosphere in hyperbolic space.
		We prove  that  umbilical hypersurfaces are only stable immersed capillary hypersurfaces whose boundary supported on a horosphere.
		Using the same method, we show that a totally geodesic hyperplane is only stable immersed type-II hypersurface whose boundary supported on a horosphere.
	\end{abstract}
	
	
	\maketitle
	
	\medskip
	
	
	\section{Introduction}
	The stability of minimal or constant mean curvature (CMC) hypersurfaces plays an important role in differential geometry of hypersurfaces. A closed CMC hypersurface is called (weakly) stable if the second variation of the area functional is nonnegative among any volume-preserving variations. A classical rigidity result proved by Barbosa-do Carmo \cite{BdC} and Barbosa-do Carmo-Eschenburg \cite{BCE} says that:
	{\it any stable immersed closed CMC hypersurfaces in a space form are geodesic spheres.}
	
	The free boundary CMC (or minimal) hypersurfaces in a domain $B$ attract recently many attentions. Here a free boundary hypersurface
	means that the hypersurface intersects its support $\partial B$ orthogonally.
	When $B$ is a $(n+1)$-dimensional unit ball $\bb^{n+1}$, there are
	many interesting results about free boundary minimal and CMC hypersufaces.
	For the  free boundary minimal hypersufaces there has been a lot of interesting work. Here we just mention the recent work of Fraser-Schoen \cite{FS1, FS2, FS3} and  refer  to  the book \cite{Hildebrandt}
	for classical results.
	In this paper we are mainly interested in
	CMC hypersufaces. There are
	several classical  rigidity results for CMC free boundary hypersurfaces, for example the Hopf type theorem by Nitsche \cite{Ni} and Ros-Souam \cite{RS} and the Alexandrov type theorem by Ros-Souam \cite{RS}.

	The simplest examples of free boundary  CMC or minimal hypersurfaces  in $\bb^{n+1}$   are the spherical caps and the geodesic disk in $\bb^{n+1}$ intersecting with $\ss^{n}$ orthogonally,
	which are in fact the unique minimizers of the area functional among all embedded hypersurfaces with a fixed enclosed volume. They are the solutions to the relative isoperimetric  problem in $\bb^{n+1}$, which was solved
	first by Burago-Mazya  \cite{BM}  and later also by Bokowsky-Sperner \cite{BS} and Almgren \cite{A} independently. A minimizer is certainly
	stable, in the sense of the nonnegativity of the second variation of the area functional under volume constraint.
	When one considers the class of immersed hypersurfaces, as in this paper,
	the concept of the stability is more suitable than the one of  the minimum, since the enclosed volume is not well defined for immersed hypersurfaces.
	The study of the  classification of stable CMC free boundary and
	capillary hypersufaces has been initiated by Ros and Vergasta \cite{RV} and Ros and Souam \cite{RS} 20 years ago.
	It has been conjectured that the free boundary totally geodesic $n$-balls and the free boundary spherical caps are only stable free boundary CMC hypersurfaces in $\bar \bb^{n+1}$. This conjecture has been recently solved by Nunes \cite{Nu} in two dimensions (see also Barbosa \cite{Ba}) and by Wang-Xia \cite{WX} in any dimensions. Moreover, Wang-Xia \cite{WX} gave complete classifications for any stable capillary hypersurfaces in a geodesic ball of any space forms. Recall that a capillary hypersurface in $\bb^{n+1}$ is a CMC hypersurface whose boundary intersects $\ss^n$ at a constant contact angle. A capillary hypersurface is called stable if the second variation of the energy functional of this hypersurface is non-negative among any volume-preserving variations.

	When $B$ is an Euclidean half-space $\rr^{n+1}_{+}$, its boundary $\p B=\R^n$ is a totally geodesic hyperplane.
	It is clear that in this case if we consider the stability problem for free boundary hypersurfaces, it reduces to the case of closed hypersurfaces mentioned above through a simply reflection. However, if one considers capillary hypersurfaces with a contact angle $\theta \not = \pi/2$, the stability problem becomes non-trivial. Very recently,
	Ainouz-Souam \cite{AS} characterized that the spherical caps are only stable immersed capillary hypersurface in $\rr^{n+1}_{+}$  in the case the contact angle $\theta<\pi/2$, under a condition that the boundary is embedded. 
	For the contact angle $\theta>\pi/2$, Choe and Koiso \cite{CK} showed the same result, under a stronger condition that the boundary of hypersurface is convex.
	See also the previous work of Marinov \cite{M2} for $n=2$.
	

	In this paper we are interested in the stability problem of
	capillary hypersurfaces  supported on a horosphere in hyperbolic space
	$\hh^{n+1}$ (of sectional curvature $-1$). A horosphere is a complete non-compact hypersurface with all principal curvatures equal to $1$.

	Our main result in this paper is the following theorem.
	\begin{theorem}\label{thm0.2}
		A compact, immersed  capillary hypersurface with boundary supported on a horosphere in $\hh^{n+1}$ is stable if and only if it is umbilical.
	\end{theorem}
We remark that an umbilical hypersurface  in $\hh^{n+1}$ is a piece of a total geodesic sphere,  an equidistant hypersurface, a horosphere or a geodesic sphere.

As a special case, we have the classification for stable free boundary constant mean curvature hypersurfaces.
	\begin{corollary}\label{cor0.1}
		A compact, immersed  free boundary CMC hypersurface with boundary supported on a horosphere in $\hh^{n+1}$ is stable if and only if it is umbilical.
	\end{corollary}

	We emphasize that in Thereom \ref{thm0.2} and hence in Corollary \ref{cor0.1} there is no requirement that the boundary is embedded. We believe that our method can be used to improve  the results of Ainouz-Souam \cite{AS}  and Choe and Koiso \cite{CK}  by removing the embeddedness of the boundary.

	Let us discuss the main difference between this work and our previous work in \cite{WX}. In order to compare, let us review the main steps used in \cite{WX}.
	The crucial ingredient in proving the classification result in \cite{WX} is a family of new Minkowski type formulas which involve no boundary term. For instance,  for an immersion smooth hypersurface $x: M\to \bar \bb^{n+1}$ with free boundary in $\rr^{n+1}$, it holds
	\begin{equation}\label{eq1111}
	n\int_M V_a dA  =\int _M H \<X_a, \nu\> d A,
	\end{equation}
	where $a\in \R^{n+1}$ is a constant vector field, $V_a$ and  $X_a$  are defined by
	\begin{equation*}
	V_a:=\<x,a\>, \qquad
	X_a:= \<x, a\>x-\frac 12 (1+|x|^2)a.
	\end{equation*}
	The key features of $X_a$ are its conformal Killing property and being tangential to the support $\ss^n$. By \eqref{eq1111} the function $nV_a-H\<X_a, \nu\>$, $a\in \R^{n+1}$ can be used as an admissible test function for the stability. Moreover this function has nice properties, which imply that the stability leads to a geometric condition
	\begin{equation}\label{1.2}
	\int_M (n |x|^2 -\frac 12  (|x|^2-1) H \langle x,\nu\rangle) (n|h|^2-H^2)dA \le 0.
	\end{equation}
	 If the integrand is non-negative, it is easy to see that the hypersurface must be umbilical.
	One can not directly to show the non-negativity. Instead we  used
	another
	auxiliary function $$\Phi =\frac 12 (|x|^2-1)H - n\<x, \nu\>$$
	and added \eqref{1.2} with $\int_M \Phi \Delta \Phi dA$, which is in fact zero. Then the new integrand is non-negative and the umbilic is easy to show.
	Such a technique of using Minkowski formulas in handling stability problems is effective, and goes back at least to the work of \cite{BdC} and \cite{BCE}. For a specific problem, the choice of  suitable admissible test functions and the	auxiliary function $\Phi$ is the key.
	  We refer to a recent survey paper \cite{WXsurvey} for details.
	


	We follow closely this approach in this paper, with  crucial modifications. We indicate the modifications  in the case of free boundary hypersurfaces.
	We shall use the half space model for $\hh^{n+1}$
	and consider  the horosphere
	\begin{equation*}
	\mathcal{H}=\{x\in\mathbb{R}^{n+1}_{+}:x_{n+1}=1\} .
	\end{equation*}
We first establish a Minkowski type formula for free boundary hypersurface supported on $\mathcal H$,
	\begin{eqnarray}\label{intro-Mink16}
	\int_M nV_{n+1}-\bar{g}(X_{n+1},\nu)H\,dA=0,
	\end{eqnarray}
	where $V_{n+1}$ and $X_{n+1}$ are defined by
	\begin{equation*}
	V_{n+1}:=\frac{1}{x_{n+1}}, \qquad
	X_{n+1}:=x-E_{n+1}.
	\end{equation*}
	Here $X_{n+1}$ is a conformal Killing vector field in $\hh^{n+1}$ parallel to $\mathcal{H}$.
	Different from the previous case, where $(n+1)$ conformal Killing fields $X_a$, $a=E_i, i=1,\cdots, n+1$ are used,  we here only use \textbf{one} conformal vector field $X_{n+1}$.\footnote{We remark that
		there are {other}  conformal Killing vector fields
		\begin{equation*}
		X_{\alpha}=x_{n+1}E_{\alpha}-x_{\alpha}E_{n+1}+(x_{\alpha}x-\frac{1}{2}|x|^{2}E_{\alpha}), \quad \alpha=1,\cdots,n,
		\end{equation*}
		from which one can also obtain Minkowski type formulas. However we are unable to use those formulas.}
	Similarly, by \eqref{intro-Mink16}
	we know \begin{eqnarray}\label{x-varphi}
	\varphi_{n+1}=nV_{n+1}-\bar{g}(X_{n+1},\nu)H
	\end{eqnarray}
	is an admissible test function  in the stability inequality, by which we  derive an integral inequality
	\begin{eqnarray}\label{xeq1}
	&&\int_{M}\left(nV_{n+1}^2+HV_{n+1}\bar{g}(E_{n+1},\nu)\right)
	(n|h|^2-H^2)
	dA-H\int_{M}\bar{g}(x,\nu)(n|h|^2-H^2)V_{n+1} dA\leq 0.
	\end{eqnarray}
As a next step we want to add some vanishing integral term $\int_M \Delta \frac12\Phi^2 (=0)$ to handle the first term in the RHS of \eqref{xeq1}. However, the second term in the RHS of \eqref{xeq1} makes trouble.
	The key observation is that the term is equal to $\int_{M}\bar{g}(x,\nu)J\varphi_{n+1}dA$, which can be transformed to a boundary integral, by utilizing the Killing property of $x$. Now
	by using a special choice $\Phi$ (see \eqref{auxi}), with $0=\int_{M}\Delta \frac12\Phi^2-\int_{\p M}\Phi \n_\mu\Phi$,
	we find $\int_{M}\Delta \frac12\Phi^2$ controls the first term in the RHS of \eqref{xeq1}, while  $\int_{\p M}\Phi \n_\mu\Phi$ cancels the boundary integral transformed from the second term in the RHS of \eqref{xeq1}.
	Moreover, the new resulted integrand has a sign and the classification follows.
	In the proof we should be careful of the
	roles that the position vector field $x$ and the constant vector field
	play. They are quite different to the roles in \cite{WX}.
	 Especially here the position vector $x$ plays a very crucial role.
	We refer to Section \ref{sec4} for the detailed proof and for the general case of capillary hypersurfaces.
	
	The  Minkowski formula \eqref{intro-Mink16}, and also its generalization \eqref{Mink1} below, has its own interest. In particular, using the idea of Guan-Li \cite{GL, GLW},  from \eqref{intro-Mink16} one can introduce a suitable flow to deform
	free boundary hypersurfaces supported on a horosphere as in \cite{SWX, WX, WWe}. With this method  we hope to establish Alexandrov-Fenchel inequalities for these hypersurfaces as in \cite{SWX} in a forthcoming paper.

	Our method also applies to type-II stable capillary hypersurfaces with boundary supported on a horosphere in $\hh^{n+1}$. A   capillary hypersurface is called type-II stable if the second variation of the area functional at a capillary hypersurface is nonnegative among all wetting-area-preserving variations (see Section \ref{sec5}). This concept is related to the type-II partitioning problem which has been considered by Burago-Maz'ya in late 60s \cite{BM} and Benkowski-Sperner \cite{BS}. Recently, Guo-Xia \cite{GX1}
	gave the complete classification for type-II stable capillary hypersurfaces in a geodesic ball in a space form.
	We observe that  $\varphi_{n+1}$ in \eqref{test-function} below also satisfies $\int_{\p M}\varphi_{n+1} ds=0$, which could be used as a test function for the type-II stability problem. By using the same proof 
	we get the following classification.
	\begin{theorem}\label{thm0.3}
		An immersed minimal hypersurface $M$ with capillary boundary supported on a horosphere in $\hh^{n+1}$ is type-II stable if and only if $M$ is totally geodesic.
	\end{theorem}

In  fluid mechanics 	a capillary surface is an interface between two fluids   in the absence of gravity.
Sessile drops, liquid bridges, rivulets, and liquid drops on fibers are  examples of capillary surfaces. Various stability problems are interested in fluid mechanics, see for example
\cite{Fluids}.
 For more  information about capillary hypersurfaces, especially the physical interpretation of capillary hypersurfaces,  we refer to the book of R. Finn \cite{Finn}.

	There have been many interesting stability results for capillary hypersurfaces within other types of domains, e.g.,
	in a wedge, a slab, a cone, a cylinder or in a polygon. Here we just mention a few
	\cite{AlS,  AS, LiXiong, LiXiong2, M1, M2, pyo,RAF2, Rosales,  Souam2, V}.

	The remaining part of this paper is organized as follows. In Section \ref{sec2} we present the basic properties of capillary hypersurfaces
	supported on a horosphere. 
	In Section \ref{sec3} we prove useful geometric formulas for hypersurfaces supported on a horosphere.
	Moreover, we find an admissible test function, by proving a Minkowski type identity.  We prove Theorem \ref{thm0.2} and Theorem \ref{thm0.3} for the capillary hypersurface  in Section \ref{sec4} and  the type-II hypersurface  in Section \ref{sec5} respectively.

\section{Capillary hypersurfaces supported on a hoposphere}\label{sec2}

Let $(\bar M^{n+1}, \bar g)$ be an oriented $(n+1)$-dimensional Riemannian manifold and $B$  a domain in $\bar M$ with smooth boundary $\p B$ in $\bar M$. Let $x: (M^n, g)\to (\bar M, \bar g)$ be an
isometric immersion of an orientable $n$-dimensional compact manifold $M$ with boundary $\p M$ satisfying $x_{|\p M}: \p M \to \p B$. Such an immersion is called an immersion supported on $\p B$.
  We emphasize that in this paper any  hypersurface we consider
  is  immersed
    and  its boundary map $x_{|\p M}: \p M \to \p B$ is also only immersed. Both are  not necessarily embedded. However, for the convenience of notation, we do not distinguish $M$ with its image $x(M)$ ($\p M$ with $x(\p M)$ resp.), through all computations are in fact carried  out on $M$ by using the pull-back of $x$.

We denote by $\bar \n$, $\bar \Delta$ and $\bar \n^2$ the gradient, the Laplacian and the Hessian on $\bar M$ w.r.t. $\bar g$ respectively, while by $\n$, $\Delta$ and $\n^2$ the gradient, the Laplacian and the Hessian on $M$ w.r.t. its induced metric respectively.
We will use the following terminology for four normal vector fields.
We choose one of the unit normal vector field along $x$ and denote it by $\nu$.
We denote  by $\bar N$ the unit outward normal to $\p B$ in $B$ and $\mu$ be the unit outward normal to $\p M$ in $M$.
Let $\bar \nu $ be the unit normal to $\p M$ in $\p B$ such that the bases $\{\nu, \mu\}$ and $\{\bar \nu  , \bar N\}$ have the same orientation in the
normal bundle of $\p M\subset \bar M$. See Figure 1, where
 $\bar M =\hh^{n+1}$ and $\p B = {\mathcal H}$, a horosphere.
\begin{figure}
	\centering
	\includegraphics[height=7cm,width=13cm]{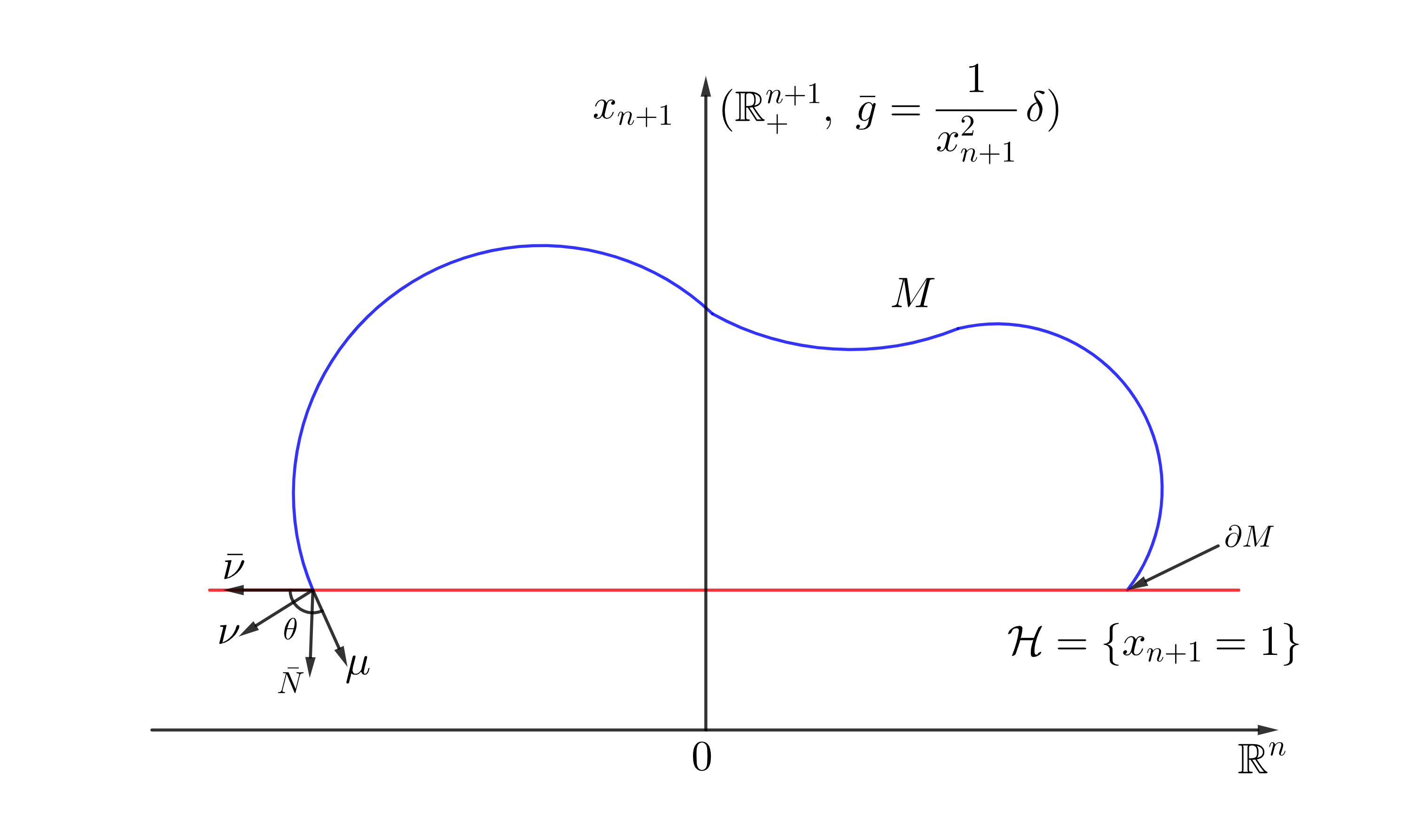}
	\caption{Hypersurface $M$  supported on  $\mathcal{H}$.}
\end{figure}


Denote by $h$ and $H$ the second fundamental form and the mean curvature of the immersion $x$ respectively. Precisely,
$h(X, Y)= \bar g(\bar \n_X \nu, Y)$ and $H=\tr_g(h).$
 The boundary map $x_{|\p M}: \p M \to \p B$ is an immersion in $\p B$. Its mean curvature is denoted by $\hat H$.




By an admissible variation of $x$ we mean a differentiable map $x: (-\epsilon, \epsilon)\times M\to\bar M$ such that $x(t, \cdot): M\to \bar M$ is an immersion satisfying $x(t, \p M)\subset\p B$ for every $t\in (-\ep, \ep)$ and $x(0, \cdot)=x$. For this variation,
the area  functional             $A: (-\ep, \ep)\to\rr$ and the volume functional $V: (-\ep, \ep)\to\rr$ are defined by
\begin{eqnarray*}
&&A(t)=\int_{ M} dA_t,\\
&&V(t)=\int_{[0,t]\times M} x^*dV_{\bar M},
\end{eqnarray*}
where $dA_t$ is the area element of $M$  with respect to the metric induced by $x(t, \cdot)$ and $dV_{\bar M}$ is the volume element of $\bar M$.
A variation is said to be volume-preserving if $V(t)=V(0)=0$ for each $t\in (-\ep, \ep)$.
Another area functional $A_W(t): (-\ep, \ep)\to\rr$, which is called wetting area functional, 
  is defined by
\begin{equation*}
  A_{W}(t)=\int_{[0,t]\times \p M} x^*dA_{\p B},
\end{equation*}
where $dA_{\p B}$ is the area element of $\p B$.
Fix a real number $\theta\in (0, \pi)$. The energy functional  $E(t): (-\ep, \ep)\to\rr$ is defined by
$$E(t)=A(t)-\cos \theta \, A_W(t).$$

The first variation formulas of $A(t)$, $A_W(t)$, $V(t)$ and $E(t)$ for an admissible variation with  a variation vector field $Y=\frac{\p}{\p t}x(t,\cdot)|_{t=0}$ are given by
\begin{eqnarray}
A'(0) &=&\int_M  H\bar g(Y, \nu) \,dA+\int_{\p M} \bar g(Y, \mu)\,ds,\label{first-var}\\
A_W'(0) &=&\int_{\partial M}\bar g(Y, \bar \nu )\, ds,\label{wet-first-var}\\
V'(0)&=&\int_M \bar g(Y, \nu)dA,\label{volume-var}\\
E'(0)&=&\int_M H\bar g(Y, \nu)dA+\int_{\p M} \bar g(Y, \mu-\cos \theta \, \bar \nu  )ds,\label{energy-var}
\end{eqnarray}
where $dA$ and $ds$ are the area element of $M$ and $\p M$ respectively. For the proof see e.g. \cite{RS, WX}.
\begin{definition}
An immersion $x: M\to \bar{M}$ with boundary $\partial M$ supported on $\p B$ is said to be {\it capillary} if it is a critical point of the energy functional $E$ for any volume-preserving variation of $x$.
\end{definition}
It follows from the above  first variation formulas \eqref{volume-var} and \eqref{energy-var} that $x$ is capillary if and only if $x$ has constant mean curvature and $\p M$ intersects $\p B$ at an angle which  equals to the constant  $\theta$.
When $\theta =\frac \pi 2$, a capillary hypersurface is a free boundary CMC hypersurface.

 For each smooth function $\vp$ on $M$ with $\int_M \vp dA=0$, there exists an admissible volume-preserving variation of $x$ with the variation vector field
 having $\varphi \nu$ as its normal part (see  \cite{RS}, page 348).
When $x$ is a capillary hypersurface, for an admissible volume-preserving variation with respect to $\vp$, the second variational formula of $E$ is given by
\begin{eqnarray}\label{stab-ineq}
&&E''(0)=-\int_M\vp(\De \vp+(|h|^2+\overline{{\rm Ric}}(\nu,\nu))\vp) dA+\int_{\p M} \vp(\n_\mu \vp-q \vp)ds.
\end{eqnarray}
Here
\begin{equation}
\label{q}
q={\csc \th}h^{\p B}(\bar \nu  , \bar \nu )+\cot \th \, h(\mu,\mu),\end{equation}
$\overline{\rm Ric}$ is the Ricci curvature tensor of $\bar M$,  
and $h^{\p B}$ is the second fundamental form of $\p B$ in $\bar M$ given by $h^{\p B}(X, Y)= \bar g(\bar \n_X \bar N, Y)$, see e.g. \cite{RS}.

\begin{definition}
A capillary hypersurface is called stable if $E''(0)\ge 0$ for all volume-preserving variations, that is, \begin{eqnarray}\label{stable1}
E''(0)\ge 0,\qquad \forall \vp\in C^{\infty}(M) \hbox{ satisfies }\int_M \varphi dA=0.
\end{eqnarray}
\end{definition}




The following proposition is a well-known and fundamental  fact for capillary hypersurfaces when $\p B$ is umbilical in $\bar M$.  See e.g. \cite[Proposition 2.1]{WX}. 
 \begin{prop}\label{Proposition1} Assume  $\p B$ is umbilical in $\bar M$. Let $x: M\to \bar{M}$ be an immersion whose boundary $\p M$ intersects $\partial B$ at a constant angle $\th$. Then $\mu$ is a principal direction of $\p M$ in $M$. Namely, $h(e, \mu)=0$ for any $e\in T(\p M)$.
\end{prop}

From now on, we consider the ambient manifold $\bar M$ to be the hyperbolic space $\hh^{n+1}$.
We make a convention on the choice of $\nu$ to be
the opposite direction of mean curvature vector so that the  mean curvature of a spherical cap  is positive. Under this convention, along $\p M$, the angle between $-\nu$ and $\bar N$ or equivalently between $\mu$ and $\bar \nu$ is equal to $\th$ (see Figure 1). To be more precise, in the normal bundle of $\p M$, we have the following relations:
 \begin{eqnarray}
&&\mu=\sin \th \, \bar N+\cos \th\, \bar \nu ,  \label{mu0}
\\&&\nu=-\cos \th \, \bar N+\sin \th\, \bar \nu  .\label{nu0}
\end{eqnarray}
Equivalently,
 \begin{eqnarray}
&&\bar{N}=\sin \th \, \mu-\cos \th\, \nu,  \label{Nbar}
\\&&\bar \nu =\cos\theta\, \mu+\sin \th\, \nu.\label{nubar}
\end{eqnarray}


 In this paper  we use the upper half-space model for the hyperbolic space $\hh^{n+1}$, which is denoted by
\begin{eqnarray}\label{half-space}
\hh^{n+1}=\{x=(x_{1}, x_{2},\cdots,x_{n+1})\in \rr^{n+1}_+: x_{n+1}>0\},\quad \bar g=\frac{1}{x_{n+1}^2}\delta.
\end{eqnarray}
A horosphere, a ``sphere" in $\hh^{n+1}$ whose centre lies at $\partial_{\infty}\mathbb{H}^{n+1}$, up to a hyperbolic isometry,  is given by the horizontal plane
\begin{equation*}
\mathcal{H}=\{x\in\mathbb{R}^{n+1}_{+}:x_{n+1}=1\}.
\end{equation*}
By choosing $\bar N=-E_{n+1}=(0,\cdots, 0, -1)$, all principal curvatures of a horosphere are $\kappa=1$. Moreover, by the Gauss equation, the induced metric on a horosphere is flat and in fact a horosphere is isometric to the $n$-dimensional Euclidean space $\rr^{n}$.

We use $x$ to denote the position vector in $\hh^{n+1}$ and $\bar{\nabla}$  the Levi-Civita connection of $\hh^{n+1}$. Let $\{E_i\}_{i=1}^{n+1}$ be the canonical basis of $\rr^{n+1}$.
We use $\<\cdot ,\cdot\>$ and $\bar g$ to denote the inner product of $\rr^{n+1}$ and $\hh^{n+1}$ respectively,  $D$ and $\bar \nabla$ to denote the Levi-Civita connection of $\rr^{n+1}$ and $\hh^{n+1}$ respectively.  Let $\bar E_i = x_{n+1} E_i$. Then $\{ \bar E_i\}_{i=1}^{n+1}$  is an orthonormal basis of $\hh^{n+1}$.
The relationship of $\bar{\nabla}$ and $D$ is given by
\begin{equation}\label{YZ}
  \bar{\nabla}_{Y}Z=D_{Y}Z-Y(\ln x_{n+1})Z-Z(\ln x_{n+1})Y+\langle Y,Z\rangle D(\ln x_{n+1}).
\end{equation}
It follows easily that
\begin{eqnarray}  \label{a1}
\bar \n_Y x &=& -\bar g (Y, \bar E_{n+1})x + \bar g (Y, x ) \bar E_{n+1}, \\ \label{a2}
\bar \n_Y  E_\a &=& -\bar g (Y, \bar E_{n+1}) E_\a + \bar g (Y, \bar E_\a) E_{n+1}, \quad \forall \a=1,2, \cdots, n, \\ \label{a3}
\bar \n_Y E_{n+1}  &=& - \frac 1{x_{n+1}} Y,
\end{eqnarray}
for any vector $Y$ in $\hh^{n+1}$, which will be used in many times.

The following simple facts  play an important role in our paper.
\begin{prop}\label{lem2.1} \

$(\rm i)$  The vector fields  \,$x$ and $\{E_\alpha\}_{\alpha=1}^{n}$ are Killing vector fields in $\hh^{n+1}$, i.e,
\begin{eqnarray}\label{xKilling}
\frac12\big( \bar g (\bar \n_i x, E_j) + \bar g (\bar \n_j x, E_i)
 ) =\frac12\big( \bar g( \bar \n_i E_{\alpha}, E_j)+
 \bar g( \bar \n_j E_{\alpha}, E_i)
 \big)=0.
\end{eqnarray}

$(\rm {ii})$\,$E_{n+1}$ is a conformal Killing vector field in $\hh^{n+1}$, i.e,
\begin{eqnarray}\label{En-Killing}
\frac12\big (\bar g( \bar \n_i E_{n+1}, E_j)+
\bar g( \bar \n_j E_{n+1}, E_i)
\big)= -\frac{1}{x_{n+1}}\bar{g}_{ij}.
\end{eqnarray}
Here $\bar \n_i = \bar \n _{E_i}$ and $\bar g_{ij}= \bar g(E_i, E_j)$.
\end{prop}
\begin{proof} These are clearly  well-known facts. For convenience of the reader, we provide the proof.
By \eqref{a1}, we have
$$ \bar g (\bar \n_i x, E_j) = -\bar g(E_i, \bar E_{n+1} ) \bar g (x, E_j) +\bar g(E_j, x) \bar g(\bar E_{n+1} , E_j),$$
which  is antisymmetry in $i$ and $j$, and hence $x$ is Killing. The proof for $E_\a$ is similar, following from \eqref{a2}.
%
From  \eqref{a3} we have
$$\bar g( \bar \n_i E_{n+1}, E_j) = -\frac 1 {x_{n+1}} \bar g_{ij}, $$
hence $E_{n+1}$ is a conformal Killing field with the conformal factor $-\frac{1}{x_{n+1}}$.
\end{proof}

Now we introduce a conformal Killing vector field $X_{n+1}$ and a function $V_{n+1}$ in $\hh^{n+1}$ that  we will use
later. Denote
  \begin{equation}\label{cfkill2}
     X_{n+1}=x-E_{n+1}, \quad V_{n+1}=\frac{1}{x_{n+1}}.
  \end{equation}
  From Proposition \ref{lem2.1} it is clear that
\begin{prop}\label{xaa}
$(\rm i)$\,$X_{n+1}$ is a conformal Killing vector field with $\frac{1}{2}\mathcal{L}_{X_{n+1}}\bar{g}=V_{n+1}\bar{g}$, namely
\begin{eqnarray}\label{XXaeq1}
\frac12\big[\bar \n_i (X_{n+1})_j+ \bar \n_j (X_{n+1})_i\big]=V_{n+1}\bar{g}_{ij}.
\end{eqnarray}
 $ ({\rm ii})$ $X _{n+1}\mid_{\mathcal{H}}$ is a tangential vector field on $\mathcal{H}$, i.e.,
\begin{eqnarray}\label{XXaeq2}
\bar{g}(X_{n+1}, \bar{N})=0 \quad\text{\rm on}\,\,\mathcal{H}.
\end{eqnarray}
\end{prop}
\begin{prop}\label{xaa2} \, $V_{n+1}$ satisfies the following properties:
\
\begin{eqnarray}\label{va2}
  \bar{\n}^2 V_{n+1}&=& V_{n+1}  \bar{g } \quad\quad\text{in}\,\, \hh^{n+1},\\
\partial_{\bar{N}}V_{n+1}&=& V_{n+1} \quad  \quad\,\,\text{on}\,\, \mathcal{H}.
\end{eqnarray}
\end{prop} \begin{proof} See \cite[Proposition 2.2]{GX2}.\end{proof}

The stability inequality (see \eqref{stab-ineq}, \eqref{q} and \eqref{stable1}) of capillary hypersurfaces  supported on $\mathcal{H}$ reduces to the following:
\begin{eqnarray}\label{stab-ineq'}
&&E''(0)=-\int_M\vp J\vp dA+\int_{\p M} \vp(\n_\mu \vp-q \vp)ds\ge 0, \qquad \forall \vp  \hbox{ with }\int_M \varphi dA=0
\end{eqnarray}
where
\begin{eqnarray}\label{jacobi}
J=\Delta+(|h|^2-n)
\end{eqnarray}
is the Jacobi operator and
\begin{equation}\label{q'}
q={\csc \th}+\cot \th \, h(\mu,\mu).\end{equation}

It is known that that any umbilical capillary hypersurfaces supported on a horosphere
are stable.
\begin{prop}\label{stable-necess} Any umbilical capillary hypersurface supported on the horosphere $\mathcal{H}$ is stable.
\end{prop}

In fact  these are minimizing among embedded hypersurfaces. For embedded hypersurfaces
one can define global functionals $A$, $V$ and  $A_W$ for the area, the enclosed and the wetting area respectively, and  the global  energy functional $E$. Therefore,
one can consider the minimizer of $E$ among the fixed enclosed volume $V$.

\

\section{Key formulae for capillary hypersurfaces supported on a horosphere }\label{sec3}
In this Section
we show useful facts  about capillary hypersurfaces $x: M\to\hh^{n+1}$ supported on  $\mathcal{H}$  that we will
use later.
For simplicity of the notation, we will omit writing the volume form $dA$ on $M$ and the area form $ds$ on $\p M$.

  We begin with  a relation between $H$ and $\hat H$, where $\hat{H}$ is the mean curvature of $\partial \S$ on $\mathcal{H}$.
\begin{prop} Along $\p M$ we have \begin{equation}
	h(\mu, \mu)=H-\sin\theta \hat{H}+(n-1)\cos\theta.\label{hmumu}
\end{equation}
\end{prop}
\begin{proof}
	By \eqref{nu0} and the fact that the principal curvature of  $\mathcal{H}$ is 1, we have
	\begin{equation*}
	h(\mu, \mu)=H-\div_{\partial M}\nu=H-\div_{\partial M}(\sin\theta\bar \nu -\cos\theta \bar{N})=H-\sin\theta \hat{H}+(n-1)\cos\theta.
	\end{equation*}
\end{proof}
\

The following Minkowski type formula will be used later.

\begin{prop}\label{prop-integral}Let $x: M\to\hh^{n+1}$ be an isometric immersion supported on $\mathcal{H}$. Assume $x(M)$ intersects $\mathcal{H}$ at a constant contact angle $\th \in (0, \pi)$. Then
\begin{eqnarray}
&&\int_{\partial M}\bar{g}(x,\bar \nu)\hat{H}-(n-1)\,ds=0\label{boundaryM}\\
&&\int_{M}\bar{g}(x,\nu)HdA =\int_{\partial M}(-\cos\theta\bar{g}(x,\bar \nu )+\sin\theta)\,ds.\label{boundary2}
\end{eqnarray}
\end{prop}

\begin{proof}
\eqref{boundaryM} is in fact the classical Minkowski formula in $\R^{n}$, since the horosphere $\mathcal H$ is isometric to $\R^{n}$ and  $\bar{g}(x,\bar \nu)|_{\mathcal H}=\<\hat{x}, \bar \nu\>$, where $\hat{x}=(x_1,\cdots , x_n, 0)$ is the position vector on $\mathcal{H}$ from the point $(0, \cdots, 0, 1)$. Therefore we only provide the proof of
\eqref{boundary2}, which uses a similar idea.


It is clear $$
\div_{M}x^{T} = \div_M (x-\bar{g}(x,\nu)\nu)=-\bar{g}(x,\nu)H.
$$
Integration by parts gives
\begin{eqnarray} \int_M -\bar{g}(x,\nu)H &=& \nonumber
\int_{M}\div_{M}x^{T} =\int_{\partial M}\bar{g}(x,\mu) \\
&=&\int_{\partial M}\bar{g}(x,\cos\theta\bar \nu +\sin\theta \bar{N})=\int_{\partial M}(\cos\theta\bar{g}(x,\bar \nu )-\sin\theta), \nonumber
\end{eqnarray}
where we have used $\bar{g}(x,\bar{N})=-1$ on $\partial M$ in the last equality.
 \end{proof}

Next we derive another crucial integral identity.
\begin{prop}\label{Proposition-x} Let $x: M\to\hh^{n+1}$ be an isometric immersion supported on $\mathcal{H}$. Assume $x(M)$ intersects $\mathcal{H}$ at a constant contact angle $\th \in (0, \pi)$. Then
	\begin{eqnarray} \label{boundary3}
	&&\int_{M}n\bar{g}(x,\nu)dA=\int_{\partial M}\bar{g}(x,\bar \nu )\,ds.
		\end{eqnarray}
\end{prop}

\begin{proof}
In order to prove \eqref{boundary3}, we consider the following vector $Z$ on $M$:
\begin{equation}\label{zn}
  Z=\bar{g}(x,\nu)  \bar{E}_{n+1}-\bar{g}(\bar{E}_{n+1},\nu)x.
\end{equation}
Recall that  $\bar{E}_{n+1}=x_{n+1}E_{n+1}$.
Along $\partial M$ we have
\begin{eqnarray}\label{zn2}
  \bar{g}(Z,\mu)&=&\bar{g}(x,\nu)\bar{g}(\bar{E}_{n+1},\mu)-\bar{g}(\bar{E}_{n+1},\nu)\bar{g}(x,\mu)\\
  &=&-\bar{g}(x,\nu)\bar{g}(\bar{N},\mu)+\bar{g}(\bar{N},\nu)\bar{g}(x,\mu)\nonumber\\
  &=&-\sin\theta\bar{g}(x,\nu)-\cos\theta\bar{g}(x,\mu)\nonumber\\
  &=&-\bar{g}(x,\bar \nu),\nonumber
\end{eqnarray}
where we have used \eqref{nubar} and the fact $\bar{N}=-\bar{E}_{n+1}$ on $\partial M$.
By integrating by parts we obtain
\begin{equation}\label{diver}
-\int_{\partial M}\bar{g}(x,\bar \nu )=\int_{\partial M}\bar{g}(Z,\mu)=\int_{M}\div_{M}(Z^{T}).
\end{equation}

Now we claim that
\begin{equation}\label{diver2}
  \div_{M}(Z^{T})=-n\bar{g}(x,\nu).
\end{equation}
Then Proposition  \ref{Proposition-x} follows from  claim \eqref{diver2}. Therefore we only need to show this claim.
 First we see that $Z$ is tangential, i.e, $\bar g (Z, \nu)=0$, which implies $ \div_{M}(Z^{T})=\div_{M}(Z)$. In view of \eqref{a1}, we know that
$$Z= \bar \n_\nu x.$$
Let $\{e_\a\}_{\a=1}^n$ be an othonormal frame of $M$. By using \eqref{xKilling} and the Riemannian curvature of $\hh^{n+1}$ being $-1$, we have
	\begin{eqnarray}\label{xmu34}
	\div_{M}(Z^{T})&=&\div_{M}(Z)
	\\&=&\bar{g}(\bar{\nabla}_{e_{\alpha}}(\bar{\nabla}_{\nu}x), e_{\alpha})\nonumber\\
	&=&\bar g(\bar \n_{\nu} (\bar \n_{e_\a} x), e_\a)-\bar g(\bar \n_{[\nu, e_\a]}x, e_\a)- \bar g(\bar R(\nu, e_\a)x, e_\a)\nonumber
	\\&=&\bar \n_{\nu}\bar g(\bar \n_{e_\a} x, e_\a)-\bar g(\bar \n_{e_\a} x,  \bar \n_{\nu} e_\a)-\bar g(\bar \n_{[\nu, e_\a]} x, e_\a)-n\bar g(x, \nu)\nonumber
	\\&=&-\bar g(\bar \n_{e_\a} x,  \bar \n_{e_\a}\nu+[\nu, e_\a])-\bar g(\bar \n_{[\nu, e_\a]} x, e_\a)- n\bar g(x, \nu)\nonumber
	\\&=&-h_{\alpha\beta}\bar g(\bar{\nabla}_{e_{\alpha}}x, e_{\beta})-\mathcal{L}_{x}\bar{g}([\nu,e_{\alpha}],e_{\alpha})-n\bar{g}(x,\nu)\nonumber
	\\&=&-n\bar g(x, \nu).\nonumber
	\end{eqnarray}
Thus we have  claim \eqref{diver2} and  finish the proof of the Proposition.
\end{proof}

As a consequence we get the following  identity.
\begin{corollary}Let $x: M\to\hh^{n+1}$ be a CMC hypersurface supported on $\mathcal{H}$. Assume $x(M)$ intersects $\mathcal{H}$ at a constant contact angle $\th \in (0, \pi)$. Then
  	\begin{eqnarray}
	&&\int_{\partial M}n\sin\theta-\bar{g}(x,\bar \nu )H-n\cos\theta \bar{g}(x,\bar \nu )\,ds=0.\label{integral-b}
	\end{eqnarray}
	\end{corollary}
	\begin{proof}Since $H$ is constant, we obtain \eqref{integral-b} from \eqref{boundary2} and \eqref{boundary3}.
	\end{proof}

Next we  can use the conformal Killing vector field $X_{n+1}$ to prove a new Minkowski type formula, which is very powerful for the study of  hypersurfaces with boundary
intersecting $\mathcal{H}$ at a constant angle in $\hh^{n+1}$.
\begin{prop}\label{Minkhorol}  Let $x: M\to  \hh^{n+1}$ be an isometric immersion supported on $\mathcal{H}$. Assume  $x(M)$ intersects $\mathcal{H}$ at a constant contact angle $\th \in (0, \pi)$. Then
\begin{eqnarray}\label{Mink1}
\int_M nV_{n+1}-\bar{g}(X_{n+1},\nu)H-n\cos \theta \bar{g}(x,\nu)\,dA=0.
\end{eqnarray}
\end{prop}
\begin{proof}
First, from \eqref{XXaeq1}, we have
\begin{equation}\label{LHS-mink}
  \div_{M}(X_{n+1}^{T})=\div_{M}(X_{n+1}-\bar{g}(X_{n+1},\nu)\nu)=nV_{n+1}-\bar{g}(X_{n+1},\nu)H.
\end{equation}
Then by integration by parts, we get
\begin{equation}\label{boundary}
  \int_{M}\div_{M}(X_{n+1}^{T})=\int_{\partial M}\bar{g}(X_{n+1},\mu)=\cos\theta\int_{\partial M}\bar{g}(X_{n+1},\bar \nu )=\cos\theta\int_{\partial M}\bar{g}(x,\bar \nu ).
\end{equation}
Here we have also used \eqref{mu0}, \eqref{XXaeq2} and $\bar{g}(E_{n+1},\bar \nu )=0$ along $\partial M$.
 Combining \eqref{LHS-mink}-\eqref{boundary} and \eqref{boundary3}, the proof is completed.
\end{proof}

Recall $J=\Delta+(|h|^2-n)$, the Jocobi operator. We next derive differential equations for geometric  quantities
$\bar g (x, \nu), $ $\bar g(E_{n+1}, \nu)$ and $ \bar g(X_{n+1}, \nu)$.
\begin{prop}\label{prop4.66} Let $x: M\to\hh^{n+1}$ be a CMC hypersurface.  Then the following identities holds:
	\begin{eqnarray}
	J \bar g(x, \nu)&=&0,\label{xmu}\\
	J\bar g(E_{n+1}, \nu)&=&-HV_{n+1}-n\bar g(E_{n+1},\nu),\label{Ennu}\\
	J\bar g(X_{n+1}, \nu)&=&HV_{n+1}-n\bar g(\bar \n V_{n+1},\nu) = HV_{n+1}+n\bar g(E_{n+1},\nu).\label{Xmu1}
	\end{eqnarray}
\end{prop}
\begin{proof}
	It is clear that  \eqref{Xmu1}  follows from \eqref{xmu} and \eqref{Ennu}. Therefore, we only need to show \eqref{xmu} and \eqref{Ennu}.
	In fact, \eqref{xmu} is a known fact, since $x$ is a Killing vector field.
	For convenience of the reader we give a direct computation.
	For a fixed $p\in M$ let $\{e_\a\}_{\a=1}^{n}$ at $p$ be the local orthonormal basis as above.
By \eqref{xKilling}, we  calculate at $p$,
	\begin{eqnarray}
	e_{\alpha}\bar{g}(x,\nu)=\bar{g}(x,\bar{\nabla}_{e_{\alpha}}\nu)+\bar{g}(\bar{\nabla}_{e_{\alpha}}x,\nu)=\bar{g}(x,\bar{\nabla}_{e_{\alpha}}\nu)-\bar{g}(\bar{\nabla}_{\nu}x, e_{\alpha}).
	\end{eqnarray}
	It follows that
	\begin{eqnarray}\label{xmu33}
	&&\Delta\bar{g}(x,\nu)=e_{\alpha}e_{\alpha}\bar{g}(x,\nu)\\
	&=&\bar{g}(\bar{\nabla}_{e_{\alpha}}x,\bar{\nabla}_{e_{\alpha}}\nu)+g(x,\bar{\nabla}_{e_{\alpha}}(\bar{\nabla}_{e_{\alpha}}\nu))-\bar{g}(\bar{\nabla}_{e_{\alpha}}(\bar{\nabla}_{\nu}x), e_{\alpha})-\bar{g}(\bar{\nabla}_{\nu}x, \bar{\nabla}_{e_{\alpha}}e_{\alpha})\nonumber\\
	&=&h_{\alpha\beta}\bar{g}(\bar{\nabla}_{e_{\alpha}}x, e_{\beta})+\bar{g}(x,\nabla H-|h|^{2}\nu)-\bar{g}(\bar{\nabla}_{e_{\alpha}}(\bar{\nabla}_{\nu}x), e_{\alpha})+H\bar{g}(\bar{\nabla}_{\nu}x, \nu)\nonumber\\
	&=&-|h|^{2}\bar{g}(x,\nu)-\bar{g}(\bar{\nabla}_{e_{\alpha}}(\bar{\nabla}_{\nu}x), e_{\alpha}),\nonumber
	\end{eqnarray}
	where we have used \eqref{xKilling}  and  the fact $H$ is constant.
		In \eqref{xmu34}, we have proved that
	\begin{eqnarray}\label{xmu34'}
	&&-\bar{g}(\bar{\nabla}_{e_{\alpha}}(\bar{\nabla}_{\nu}x), e_{\alpha})=n\bar g(x, \nu).
	\end{eqnarray}
\eqref{xmu}	follows from  \eqref{xmu33} and \eqref{xmu34'}.
	
	Using \eqref{a3} and \eqref{En-Killing}, we can check that
	\begin{eqnarray}\label{xmu35}
	\Delta\bar{g}(E_{n+1},\nu)&=&e_{\alpha}e_{\alpha}\bar{g}(E_{n+1},\nu)
	=e_{\alpha}(\bar{g}(\bar{\nabla}_{e_{\alpha}}E_{n+1},\nu)+\bar{g}(E_{n+1},\bar{\nabla}_{e_{\alpha}}\nu))\nonumber\\
	&=&e_{\alpha}\bar{g}(E_{n+1},\bar{\nabla}_{e_{\alpha}}\nu)
	=\bar{g}(\bar{\nabla}_{e_{\alpha}}E_{n+1}, \bar{\nabla}_{e_{\alpha}}\nu)+\bar{g}(E_{n+1}, \bar{\nabla}_{e_{\alpha}}(\bar{\nabla}_{e_{\alpha}}\nu))\nonumber\\
	&=&h_{\alpha\beta}\bar{g}(\bar{\nabla}_{e_{\alpha}}E_{n+1}, e_{\alpha})+\bar{g}(E_{n+1}, \nabla H-|h|^{2}\nu)\nonumber\\
	&=&-H V_{n+1}-|h|^{2}\bar{g}(E_{n+1}, \nu),\nonumber
	\end{eqnarray}
	which implies \eqref{Ennu}.
\end{proof}
\

 Now we check the boundary equations of corresponding geometric quantities.
\begin{prop}\label{prop-boundary} Let $x: M\to  \hh^{n+1}$ be an isometric immersion supported on $\mathcal{H}$. Assume  $x(M)$ intersects $\mathcal{H}$ at a constant contact angle $\th \in (0, \pi)$. Then along $\p M$, we have
\begin{eqnarray}\label{Hyp-dddd1}
{\nabla}_{\mu}(V_{n+1}-\cos\theta\bar{g}(E_{n+1},\nu))&=&q(V_{n+1}-\cos\theta\bar{g}(E_{n+1},\nu))\label{boundary-111}\\
{\nabla}_{\mu}\bar{g}(X_{n+1},\nu)&=&q\bar{g}(X_{n+1},\nu)\label{boundary-222}\\
{\nabla}_{\mu}\bar{g}(x,\nu)&=&\bar{g}(x, \bar \nu )+h(\mu, \mu)\bar{g}(x, \mu),\label{Hyp-dddd22}
\end{eqnarray}
where $q$ is defined in \eqref{q'}.
\end{prop}
\begin{proof}
In this proof we always take value along $\partial M$ and use \eqref{mu0} and \eqref{nu0}.

By directly calculating, we have
\begin{eqnarray}\label{X-boun}
&&\,\,q(V_{n+1}-\cos\theta\bar{g}(E_{n+1},\nu))=\left({\csc\theta}+\cot\theta h(\mu,\mu)\right)(1-\cos^{2}\theta)=\sin\theta(1+\cos\theta h(\mu, \mu)),
\end{eqnarray}
From \eqref{a3} and Proposition \ref{Proposition1}, we have
\begin{eqnarray}\label{X-boun2}
\quad\quad\nabla_{\mu}V_{n+1}-\cos\theta\nabla_{\mu}\bar{g}(E_{n+1},\nu)&=&-\bar{g}(E_{n+1},\mu)-\cos\theta(\bar{g}(\bar{\nabla}_{\mu}E_{n+1}, \nu)+\bar{g}(E_{n+1},\bar{\nabla}_{\mu}\nu))\\
&=&-\bar{g}(E_{n+1},\mu)(1+\cos\theta h(\mu, \mu))\nonumber\\
&=&\sin\theta(1+\cos\theta h(\mu, \mu)),\nonumber
\end{eqnarray}
where we use $\bar{g}(E_{n+1},\mu)=-\sin\theta$ on $\partial M$.
 It is clear that equation \eqref{boundary-111} follows from \eqref{X-boun} and \eqref{X-boun2}.

Using the fact that \eqref{Nbar} and \eqref{XXaeq2}, we have
\begin{equation}\label{relation}
  \bar{g}(X_{n+1},\mu)=\cot\theta \bar{g}(X_{n+1},\nu).
\end{equation}

From \eqref{relation}, \eqref{YZ} and Proposition \ref{Proposition1}, we get
\begin{eqnarray*}
&&{\nabla}_{\mu}\bar{g}(X_{n+1},\nu) \\
&=& \bar{g}(\bar{\nabla}_{\mu}X_{n+1},\nu)+\bar{g}(X_{n+1},\bar{\nabla}_{\mu}\nu)\\
&=&\bar{g}(\bar{\nabla}_{\mu}X_{n+1},\nu)+\bar{g}(X_{n+1},\mu)h(\mu,\mu)\\
&=&-\bar{g}(\bar{E}_{n+1},\mu)\bar{g}(X_{n+1},\nu)+\bar{g}(X_{n+1},\mu)\bar{g}(\bar{E}_{n+1},\nu)+\bar{g}(X_{n+1},\mu)h(\mu,\mu)\\
&=&\sin\theta \bar{g}(X_{n+1},\nu)+\cos\theta \bar{g}(X_{n+1},\mu)+\bar{g}(X_{n+1},\mu)h(\mu,\mu)\\
&=&\sin\theta \bar{g}(X_{n+1},\nu)+\cos\theta\cot\theta \bar{g}(X_{n+1},\nu)+\cot\theta\bar{g}(X_{n+1},\nu)h(\mu,\mu)\\
&=&q\bar{g}(X_{n+1},\nu)
\end{eqnarray*}
and
\begin{eqnarray}
{\nabla}_{\mu}\bar{g}(x,\nu)&=&\bar{g}(\bar{\nabla}_{\mu}x,\nu)+\bar{g}(x,\bar{\nabla}_{\mu}\nu)\nonumber\\
&=&\bar{g}(\bar{\nabla}_{\mu}x,\nu)+h(\mu,\mu)\bar{g}(x,\mu)\nonumber\\
&=&-\bar{g}(\bar{E}_{n+1},\mu)\bar{g}(x,\nu)+\bar{g}(\bar{E}_{n+1},\nu)\bar{g}(x,\mu)+h(\mu,\mu)\bar{g}(x,\mu)\nonumber\\
&=&\sin\theta\bar{g}(x,\nu)+\cos\theta\bar{g}(x,\mu)+h(\mu,\mu)\bar{g}(x,\mu)\nonumber\\
&=&\bar{g}(x,\bar \nu)+h(\mu,\mu)\bar{g}(x,\mu),\nonumber
\end{eqnarray}
where we have used  $\bar \nu=\sin\theta\,\nu+\cos\theta\,\mu$ along $\partial M$
in the last equality.

\end{proof}

\

\section{Uniqueness for stable capillary hypersurfaces 
}\label{sec4}
Now we start to show the uniqueness result for stable capillary hypersurfaces supported on $\mathcal H$. First by the new Minkowski type formula \eqref{Mink1},
 we have an admissible test function defined by
 \begin{eqnarray}\label{test-function}
 \varphi_{n+1}:=nV_{n+1}-\bar{g}(X_{n+1},\nu)H-n\cos \theta \bar{g}(x,\nu).
 \end{eqnarray}
\begin{prop}\label{prop-3.1}Let $x: M\to\hh^{n+1}$ be a CMC hypersurface with boundary supported on $\mathcal{H}$. Assume $x(M)$ intersects $\mathcal{H}$ at a constant contact angle $\th \in (0, \pi)$.
Then $\varphi_{n+1}$ satisfies
	\begin{eqnarray}J\varphi_{n+1} &=&(n|h|^2-H^2)V_{n+1},\label
	{varphi1}  \\
	\label{Hyp-bdy1}
	{\nabla}_{\mu}\varphi_{n+1}&=&q\varphi_{n+1},\\
\int_{M}\varphi_{n+1}\,dA&=&0,\label{bdy-zero1}\\
\int_{\partial M}\varphi_{n+1}\,ds&=&0.\label{bdy-zero2}
\end{eqnarray}

\end{prop}
\begin{proof}
	\eqref{varphi1} follows from Proposition \ref{prop4.66} and \eqref{Hyp-bdy1}  from Proposition \ref{prop-boundary}.
	\eqref{bdy-zero1} is exactly \eqref{Mink1}.
Using \eqref{nu0} and \eqref{XXaeq2} on $\partial M$, we have
\begin{eqnarray}\label{X-inte}
\varphi_{n+1}&=&nV_{n+1}-\bar{g}(X_{n+1},\nu)H-n\cos \theta \bar{g}(x,\nu)\\
&=&n(1-\cos^{2}\theta)-\sin\theta \bar{g}(X_{n+1},\bar \nu )H-n\cos\theta\sin\theta \bar{g}(x,\bar \nu )\nonumber\\
&=&\sin\theta(n\sin\theta-\bar{g}(x,\bar \nu )H-n\cos\theta \bar{g}(x,\bar \nu)).\nonumber
\end{eqnarray}
Therefore, \eqref{bdy-zero2} follows from  \eqref{integral-b}.
\end{proof}

Now we are ready to prove the uniqueness result for stable capillary hypersurfaces.
\begin{theorem}\label{thmm4.1}
Let $x: M\to\hh^{n+1}$ be a stable capillary hypersurface with boundary supported on $\mathcal{H}$.  Then $x(M)$ is umbilical.
\end{theorem}
\begin{proof} 
From \eqref{bdy-zero1}, we know that $\varphi_{n+1}$ is an admissible test function in \eqref{stab-ineq'}.
Therefore, by \eqref{Hyp-bdy1}, we have
\begin{eqnarray}\label{xeq1111}
0 &\le &-\int_M\varphi_{n+1}J \varphi_{n+1}\,+\int_{\p M} \varphi_{n+1}(\n_\mu \varphi_{n+1}-q \varphi_{n+1})\,
\\&= &-\int_M (nV_{n+1}-\bar{g}(x-E_{n+1},\nu)H-n\cos\theta \bar{g}(x,\nu))J\varphi_{n+1}\,  \nonumber\\
&=&-\int_M (nV_{n+1}+\bar{g}(E_{n+1},\nu)H)J\varphi_{n+1}\, +(H+n\cos\theta)\int_{M}\bar{g}(x,\nu)J\varphi_{n+1}.\nonumber
\end{eqnarray}
We compute the last term of \eqref{xeq1111} by Green's  formula. By \eqref{xmu} and \eqref{Hyp-bdy1}, we have
\begin{eqnarray}\label{xeq11120}
\int_{M}\bar{g}(x,\nu)J\varphi_{n+1}
&=&\int_{M}J\bar{g}(x,\nu)\varphi_{n+1}+\int_{\partial M}\bar{g}(x,\nu)\nabla_{\mu}\varphi_{n+1}-\varphi_{n+1}\nabla_{\mu}\bar{g}(x,\nu)\\
&=&\int_{\partial M}(q\cdot\bar{g}(x,\nu)-\nabla_{\mu}\bar{g}(x,\nu))\varphi_{n+1}.\nonumber
\end{eqnarray}
By \eqref{Hyp-dddd22}, \eqref{nu0} and \eqref{Nbar}, we see along $\partial M$
\begin{eqnarray}\label{xeq1113}
q\cdot\bar{g}(x,\nu)-\nabla_{\mu}\bar{g}(x,\nu)&=&\left(\csc\theta+\cot\theta h(\mu,\mu)\right)\bar{g}(x,\nu)-(\bar{g}(x,\bar \nu )+h(\mu, \mu)\bar{g}(x,\mu))\\
&=&-\cot\theta \bar{g}(x,\bar{N})-{\csc\theta} \bar{g}(x,\bar{N})h(\mu,\mu)\nonumber\\
&=&\cot\theta+{\csc\theta}  h(\mu,\mu),\nonumber
\end{eqnarray}
in the last equality we have used $\bar{g}(x,\bar{N})=-1$ on $\partial M$.
Instituting \eqref{xeq11120}-\eqref{xeq1113} into \eqref{xeq1111}, we get
\begin{equation}\label{stability-2}
  \int_M (nV_{n+1}+\bar{g}(E_{n+1},\nu)H)J\varphi_{n+1}\, -(H+n\cos\theta)\int_{\partial M}(\cot\theta+{\csc\theta}  h(\mu,\mu))\varphi_{n+1}\leq0.
\end{equation}

In order to get the information from  \eqref{stability-2}, we introduce an auxiliary function
\begin{equation}\label{auxi}
  \Phi:=-H V_{n+1}-n\bar{g}(E_{n+1},\nu).
\end{equation}
By \eqref{va2} and \eqref{Ennu}, we obtain
\begin{equation}\label{auxi2}
  \Delta\Phi=(n|h|^{2}-H^{2})\bar{g}(E_{n+1},\nu).
\end{equation}
Note from \eqref{nu0} that  \begin{eqnarray}\label{Phi-boun1}
\Phi|_{\p M}=-H-n\cos\theta
\end{eqnarray}
Also, similar as in the derivation of  \eqref{X-boun2},  we can calculate
\begin{eqnarray}\label{Phi-boun2}
&&  \nabla_{\mu}\Phi
=-\sin\theta(H-nh(\mu,\mu)).
\end{eqnarray}
Inserting \eqref{auxi}- \eqref{Phi-boun2} into the following identity \begin{eqnarray}\label{xeq2}
&&\int_{M}\Phi\Delta\Phi+|\nabla \Phi|^{2} =\int_{M}\frac{1}{2}\Delta\Phi^{2}=\int_{\partial M}\Phi\nabla_{\mu}\Phi,
\end{eqnarray}
we get an integral identity
\begin{eqnarray}\label{xeq211}
&&\int_{M}(-HV_{n+1}-n\bar{g}(E_{n+1},\nu))(n|h|^{2}-H^{2})\bar{g}(E_{n+1},\nu)+|\nabla \Phi|^{2} \\
&=&(H+n\cos\theta)\sin\theta\int_{\partial M}(H-nh(\mu,\mu)).\nonumber
\end{eqnarray}
Adding \eqref{xeq211} to \eqref{stability-2} and applying \eqref{varphi1} we have
\begin{eqnarray}\label{xeq2112}
0&\ge &\int_{M}n\bar{g}(E_{n+1}^{T},E_{n+1}^{T})(n|h|^{2}-H^{2})+|\nabla \Phi|^{2} \nonumber\\
&&+(H+n\cos\theta)\int_{\partial M}\left[\sin\theta(- H+n  h(\mu,\mu))-(\cot\theta+{\csc\theta}  h(\mu,\mu))\right]\varphi_{n+1}.
\end{eqnarray}

Now we have a key observation  that the boundary term in the right hand side of \eqref{xeq2112} vanishes.
In fact, by \eqref{X-inte} we can simplify
\begin{eqnarray}\label{xeq2133}
&& \int_{\partial M}\left[\sin\theta(- H+n  h(\mu,\mu))-(\cot\theta+{\csc\theta}  h(\mu,\mu))\right]\varphi_{n+1}\\
&=&(H+n\cos\theta)\int_{\partial M}(-\sin\theta+\cos\theta \bar{g}(x,\bar \nu)+h(\mu,\mu)\bar{g}(x,\bar \nu )).\nonumber
\end{eqnarray}
Utilizing \eqref{hmumu}, \eqref{boundaryM} and \eqref{integral-b}, we obtain
\begin{eqnarray}\label{xeq2134}
&&\int_{\partial M}(-\sin\theta+\cos\theta \bar{g}(x,\bar \nu)+h(\mu,\mu)\bar{g}(x,\bar \nu ))\\
&=&\int_{\partial M}(-\sin\theta+\cos\theta \bar{g}(x,\bar \nu )+(H-\sin\theta \hat{H}+(n-1)\cos\theta)\bar{g}(x,\bar \nu ))\nonumber\\
&=&-\int_{\partial M} \sin\theta+\int_{\partial M}(H+n\cos\theta)\bar{g}(x,\bar \nu )-\int_{\partial M}\sin\theta\bar{g}(x,\bar \nu )\hat{H}\nonumber\\
&=&-\sin\theta\,|\partial M| +n\sin\theta\,|\partial M|-(n-1)\sin\theta\,|\partial M|\nonumber\\
&=&0.\nonumber
\end{eqnarray}
Therefore, inserting \eqref{xeq2133}-\eqref{xeq2134} into \eqref{xeq2112}, we get
\begin{equation}\label{umbilical}
\int_{M}n\bar{g}(E_{n+1}^{T}, E_{n+1}^{T})(n|h|^{2}-H^{2})+|\nabla\Phi|^{2}dA\leq0.
\end{equation}It follows that
\begin{eqnarray}\label{conclusion1}
\bar{g}(E_{n+1}^{T}, E_{n+1}^{T})(n|h|^{2}-H^{2})=0,
\end{eqnarray}
and $\Phi$ is a constant which implies $\Delta \Phi=0$. It follows from \eqref{auxi2} that \begin{eqnarray}\label{conclusion2}
(n|h|^{2}-H^{2})\bar{g}(E_{n+1},\nu)=0,
\end{eqnarray}
we conclude from \eqref{conclusion1} and \eqref{conclusion2} that $n|h|^{2}-H^{2}\equiv0$, i.e.,  $M$ is umbilical.
\end{proof}

\

\section{Uniqueness for type-II stable hypersurfaces supported on a horosphere}\label{sec5}

In this section we sketch the rigidity result for type-II stable hypersurfaces supported on $\mathcal{H}$ in $\hh^{n+1}$.
First we recall some basic concepts and refer to \cite{GX1} for details.
\begin{definition}
	An immersion $x: M\rightarrow \hh^{n+1}$ supported on  $\mathcal{H}$ is said to be {\it type-II stationary} if it is a stationary point of the area functional  for any wetting-area-preserving admissible variation of $x$. A type-II stationary hypersurface is said to be {\it type-II stable} if the second variation of
the area functional  is nonnegative among any wetting-area-preserving admissible variations.\end{definition}

From the  first variation formulas \eqref{first-var} and \eqref{wet-first-var}, we know that $x$ is type-II stationary if and only if $x$ is a minimal immersion, namely $H=0$, and $x(M)$ intersects $\mathcal{H}$ at a constant angle $\theta\in (0, \pi)$. When $x$ is a type-II stable hypersurface, then \begin{eqnarray}\label{stable-second}
&&A''(0)=-\int_M\vp J\vp\, dA+\int_{\p M} \vp(\n_\mu \vp-q \vp)\,ds\ge 0, \qquad \varphi\in C^\infty(M) \hbox{ with }\int_{\p M}\vp ds=0.
\end{eqnarray}
Here $J$ and $q$ are also given in \eqref{jacobi} and \eqref{q'}. See \cite[Proposition 2.2]{GX1}.
It can be shown that
for any $\vp\in C^\infty(M)$ satisfying $\int_{\p M}  \vp \,ds=0$, there exists an admissible wetting-area-preserving variation of $x$ with variational vector field having $\vp\nu$ as its normal part, see \cite[Proposition 2.1]{GX1}.

From the proof of Proposition \ref{stable-necess}, we see that any totally geodesic type-II hypersurface supported on $\mathcal{H}$ is also type-II stable.
Now we prove the uniqueness result for type-II stable hypersurfaces with boundary supported on $\mathcal{H}$  in $\hh^{n+1}$ as follows
\begin{theorem}
Let $x: M\to\hh^{n+1}$ be a type-II stable hypersurface with boundary supported on $\mathcal{H}$. Then $x(M)$ is  totally geodesic.
\end{theorem}
\begin{proof} We use $\vp_{n+1}$ as defined in \eqref{test-function}.
From \eqref{bdy-zero2}, we know that $\varphi_{n+1}$ is an admissible test function in \eqref{stable-second}. Since stability inequality \eqref{stable-second} is the same as the capillary hypersurface case,
we see that the proof in this case follows the same line with that of Theorem \ref{thmm4.1} to get that $x(M)$ is umbilical. Since $M$ is minimal, we get $x(M)$ is totally geodesic.

\end{proof}

The paper has been submitted and the results have been reported by Chao Xia in working seminars at Beihang University on September 17th, 2021 and at Hubei University on October 12th, 2021.
\


\end{document}